\documentclass[11pt]{article}
\usepackage{amsthm, amsmath, amssymb, amsfonts, url, booktabs, tikz, setspace, fancyhdr, enumerate}
\usepackage[margin = 1in]{geometry}

\usepackage[czech,english]{babel} 

\usepackage{hyperref}

\usepackage{cleveref}

\definecolor{refkey}{gray}{.75}
\definecolor{labelkey}{gray}{.2}

\usepackage[textsize=scriptsize,colorinlistoftodos]{todonotes}

\newtheorem{theorem}{Theorem}[section]
\newtheorem{proposition}[theorem]{Proposition}
\newtheorem{lemma}[theorem]{Lemma}
\newtheorem{corollary}[theorem]{Corollary}
\newtheorem{conjecture}[theorem]{Conjecture}

\theoremstyle{definition}

\newtheorem{example}[theorem]{Example}

\theoremstyle{remark}

\newcommand{\x}{\times}

\newcommand{\G}{\mathcal{G}}

\newcommand{\EE}{\mathbb{E}}

\newcommand{\RR}{\mathbb{R}}
\newcommand{\NN}{\mathbb{N}}

\newcommand{\HH}{\mathcal{H}}

\newcommand{\FF}{\mathcal{F}}
\newcommand{\FFF}{\mathbb{F}}

\newcommand{\C}{\mathcal{C}}

\newcommand{\K}{\mathcal{K}}
\newcommand{\LL}{\mathcal{L}}

\newcommand{\Q}{\mathcal{Q}}

\newcommand{\TT}{\mathcal{T}}

\newcommand{\ft}{\widehat}

\newcommand{\mockalph}[1]{}

\tikzstyle{p}+=[fill=black, circle, minimum width = 1pt, inner sep =
1pt]
\tikzstyle{w}+=[fill=white, draw, circle, minimum width = 1pt, inner sep =
1.5pt]

\tikzset{every loop/.style={}}

\title{Around the positive graph conjecture}
\author{
David Conlon\thanks{Department of Mathematics, California Institute of Technology, United States. Email:{\tt dconlon@caltech.edu}. Research supported by NSF Awards DMS-2054452 and DMS-2348859.
}
\and
Joonkyung Lee\thanks{Department of Mathematics, Yonsei University,  South Korea. Email: {\tt joonkyunglee@yonsei.ac.kr}. Research supported by Samsung STF Grant SSTF-BA2201-02, the Yonsei University Research Fund 2023-22-0125 and the National Research Foundation of Korea (NRF) grant MSIT NRF-2022R1C1C1010300.
}
\and
Leo Versteegen\thanks{Department of Pure Mathematics and Mathematical Statistics, University of Cambridge, United Kingdom. Email: {\tt lvv23@dpmms.cam.ac.uk}.
}
}
\date{}

\begin{document}
\maketitle

\begin{abstract}
A graph $H$ is said to be positive if the homomorphism density $t_H(G)$ is non-negative for all weighted graphs $G$.  
The positive graph conjecture proposes a characterisation of such graphs, saying that a graph is positive if and only if it is symmetric, in the sense that it is formed by gluing two copies of some subgraph along an 
independent set. We prove several results relating to this conjecture. First, we make progress towards the conjecture itself by showing that any connected positive graph must have a vertex of even degree. We then make use of this result to identify some new counterexamples to the analogue of Sidorenko’s conjecture for hypergraphs. In particular, we show that, for $r$ odd, every $r$-uniform tight cycle is a counterexample, generalising a recent result of Conlon, Lee and Sidorenko that dealt with the case $r = 3$. Finally, we relate the positive graph conjecture to the emerging study of graph codes by showing that any positive graph has vanishing graph code density, thereby improving a result of Alon who proved the same result for symmetric graphs. Our proofs make use of a variety of tools and techniques, including the properties of independence polynomials, hypergraph quasirandomness and discrete Fourier analysis.
\end{abstract}

\section{Introduction}

Given a graph $H$, the \emph{homomorphism density} $t_H(W)$ of the graph $H$ in the kernel, i.e., the bounded measurable symmetric function, $W:[0,1]^2\rightarrow\mathbb{R}$ is defined by 
\[t_H(W)  = \int_{[0,1]^{V(H)}} \prod_{ij \in E(H)} W(x_i, x_j) \ d\mu^{V(H)},\]
where $\mu$ is the Lebesgue measure on $[0,1]$. 
Our concern here will be with a fundamental question about this functional first raised by Lov\'asz~\cite{L08}, namely, which graphs $H$ are \emph{positive} in the sense that  $t_H(W)\geq 0$ for all kernels $W$? 

One class of positive graphs $H$ comes from a special type of involutary automorphism. 
An automorphism $\phi$ of $H$ is a \emph{cut involution} if there exists a partition $L\cup R\cup F$ of $V(H)$ with $F$ a vertex cut separating $L$ and $R$ such that $\phi$ maps $L$ and $R$ to each other and fixes $F$. If the fixed point set $F$ contains no edges of $H$ (i.e., it is a stable set), then the cut involution $\phi$ is said to be a \emph{stable involution}.
If $H$ has such a stable involution, a simple application of the Cauchy--Schwarz inequality shows that $t_H(W)\geq t_{H[L\cup F]}(W)^2$ for all kernels $W$, so the existence of a stable involution automatically implies positivity.
The \emph{positive graph conjecture}~\cite{ACHLL12} states that the converse is also true.

\begin{conjecture}[Conjecture 1 in \cite{ACHLL12}]\label{conj:pgc} A graph $H$ is positive if and only if it has a stable involution.
\end{conjecture}

This conjecture can be seen as an analogue for $H$-densities of kernels of Hilbert’s 17th problem, which asked if any non-negative polynomial can be expressed as a sum of rational squares. To bear out the analogy, we note that the positive graph conjecture is just asking if every positive graph is a `square’ formed by gluing two copies of some subgraph along a (possibly empty) independent set. Some evidence that proving the conjecture may be delicate comes from a result of Hatami and Norine~\cite{HN11} stating that homomorphism inequalities of the form $\alpha_1 t_{H_1}(W)+\cdots+\alpha_k t_{H_k}(W)\geq 0$ for some $\alpha_1,\ldots,\alpha_k\in \mathbb{R}$ and graphs $H_1, \dots, H_k$ that hold for all graphons $W$ (i.e., all non-negative kernels) are not necessarily representable as sums of squares. In fact, in this level of generality, the problem of deciding whether a given linear combination of homomorphism densities is always non-negative becomes undecidable (see also~\cite{BRW24} for a more recent undecidability result for the kernel case).

Perhaps because of such difficulties, there has been very little progress on the positive graph conjecture since it was first stated more than a decade ago. The original paper on the subject, by Antol\'in Camarena, Cs\'oka, Hubai, Lippner and Lov\'asz~\cite{ACHLL12}, contains a number of interesting partial results, including a proof of the conjecture for trees. Since then, the only significant progress was in a paper of Sidorenko~\cite{Sid22}, who relates the conjecture to its natural hypergraph generalisation. This result, which plays an important role in our proofs, will be discussed in more detail below. We first state a result which makes one more small step towards the conjecture.

\begin{theorem}\label{thm:even_deg}
    Every connected positive graph contains a vertex of even degree.
\end{theorem}

To see how this connects to~\Cref{conj:pgc}, observe that if the conjecture is true for some graph $H$, then every vertex in the fixed point set $F$ of the stable involution of $H$ must have even degree, so, provided $H$ is connected, at least one vertex has even degree. \Cref{thm:even_deg} shows that this conclusion holds for positive graphs without first proving the conjecture. Perhaps surprisingly, the proof of~\Cref{thm:even_deg} makes use of some properties of independence polynomials.

To state the aforementioned result of Sidorenko~\cite{Sid22} more precisely, let us first generalise the homomorphism density to $r$-uniform hypergraphs (or $r$-graphs) $\HH$ and bounded measurable symmetric functions $W:[0,1]^r\rightarrow\mathbb{R}$, which we call \emph{$r$-kernels}, by setting 
\begin{align*}
    t_\HH(W)  = \int_{[0,1]^{v(\HH)}} \prod_{i_1i_2\cdots i_r \in E(\HH)} W(x_{i_1}, x_{i_2},\ldots,x_{i_r}) \ d\mu^{v(\HH)}.
\end{align*}
As for graphs, an $r$-graph $\HH$ is \emph{positive} if $t_\HH(W)\geq 0$ for all $r$-kernels $W$. For $r$ odd, Sidorenko's result establishes a close connection between the positivity of $r$-graphs and that of the corresponding Levi graph, where, given an $r$-graph $\HH$, its \emph{Levi graph} $\LL(\HH)$ is its bipartite edge-vertex incidence graph. That is, $\LL(\HH)$ has vertex set $V(\HH)\cup E(\HH)$ with $\{v,e\}$ an edge if and only if $v\in e$.

\begin{theorem}[Theorem~1.2 in~\cite{Sid22}]
    If $r\geq 3$ is an odd integer and $\HH$ is an $r$-graph, then $\HH$ is positive if and only if its Levi graph $\LL(\HH)$ is positive.
\end{theorem}

Combining this with our~\Cref{thm:even_deg}, we obtain the following corollary.

\begin{corollary}\label{cor:odd}
    If $r\geq 3$ is an odd integer and $\HH$ is an $r$-graph where every vertex in $\HH$ has odd degree, then $\HH$ is not positive.
\end{corollary}

In particular, this result applies to $\C_{\ell}^{(r)}$, the $r$-uniform \emph{tight cycle} with vertex set $\{1, 2, \dots, \ell\}$ and edges $\{i, i+1, \dots, i+r-1\}$ for all $i = 1, 2, \dots, \ell$, where addition is taken mod $\ell$, showing that, for any $r \geq 3$ odd and any $\ell\geq r$, $\C_{\ell}^{(r)}$ is not positive.

Despite its seeming simplicity, \Cref{cor:odd} already allows us to prove some new results about the hypergraph generalisation of Sidorenko's conjecture. For graphs, this celebrated conjecture~\cite{Sid92,Sid93}, which remains wide open (see, for example,~\cite{CL21} and its references), states that $t_H(W)\geq t_{K_2}(W)^{e(H)}$ for any bipartite graph $H$ and any graphon $W$. The natural hypergraph analogue, which states that $t_\HH(W)\geq t_{\K_r}(W)^{e(\HH)}$ for any $r$-partite $r$-graph $\HH$ and any $r$-graphon $W$ (i.e., any non-negative $r$-kernel), is false, as was shown already by Sidorenko~\cite{Sid93}.

However, it is still interesting to ask which $r$-graphs satisfy 
Sidorenko's conjecture. For the sake of brevity, let us say that an $r$-graph $\HH$ is \emph{Sidorenko} if $t_\HH(W)\geq t_{\K_r}(W)^{e(\HH)}$ for all $r$-graphons $W$. Very recently, Conlon, Lee and Sidorenko~\cite{CLS24} showed that if an $r$-partite $r$-graph $\HH$ is not Sidorenko, then it is possible to improve the lower bound for the \emph{extremal number} $\mathrm{ex}(n,\HH)$ of $\HH$, the maximum number of edges in an $\HH$-free $r$-graph with $n$ vertices, by a polynomial factor over the classical bound coming from the probabilistic deletion method. In turn, by a result of Ferber, McKinley and Samotij~\cite{FMS20}, this implies an optimal bound on the number of $\HH$-free $r$-graphs with $n$ vertices along an infinite sequence of values of $n$. 

Motivated by these applications, several new examples of non-Sidorenko hypergraphs were found in~\cite{CLS24}, including all linear $r$-graphs of odd girth and all $3$-uniform tight cycles. Here, by combining~\Cref{cor:odd} with extra ideas involving quasirandomness, we extend this latter result by showing that $r$-uniform tight cycles are not Sidorenko for any odd $r$. 

\begin{theorem}\label{thm:cycle}
    If $r \ge 3$ is odd and $\ell > r$, then $\C_{\ell}^{(r)}$ is not Sidorenko.
\end{theorem}

Our methods also allow us to give a short alternative proof of another result from~\cite{CLS24}, namely, that the \emph{grid $r$-graph} $\G_r$ with vertex set $[r]^2$ where each row $\{(i,j):j\in[r]\}$ and each column $\{(i,j):i\in[r]\}$ forms an edge is not Sidorenko for odd $r$. The result proved in~\cite{CLS24} is actually somewhat stronger, saying that $\G_r$ is not \emph{common} for odd $r$. However, since the question of whether $\G_3$ is Sidorenko was raised by Gowers and Long~\cite{GL}, we felt that it was worth noting the simpler proof.

\begin{theorem}\label{thm:grid}
     If $r\geq 3$ is odd, then the grid $r$-graph $\G_r$ is not Sidorenko.
\end{theorem}

Our last main result connects~\Cref{conj:pgc} to a problem of Alon~\cite{A24} that arose from the emerging study of graph codes (see~\cite{AGKMS23} for much more on questions of this type). Let $d_H(n)$ be the maximum cardinality of a family of graphs on $[n]$ without two members whose symmetric difference is a copy of $H$, where we normalise by dividing out a factor of  $2^{\binom{n}{2}}$, the total number of labelled graphs on $n$ vertices. The particular problem raised by Alon asks for a classification of those graphs $H$ for which $d_H(n)=o(1)$. He suggested that perhaps $d_H(n)=o(1)$ for any graph with an even number of edges and observed that this is indeed true for linear codes (see also~\cite{V24}). He also showed that every graph $H$ with a stable involution satisfies $d_H(n)=O_H(1/n)$. Here we show that the same conclusion holds for all positive graphs, thereby generalising Alon's result.

\begin{theorem}\label{thm:code}
    Every positive graph $H$ satisfies $d_H(n)=O_H(1/n)$.
\end{theorem}

Of course, if~\Cref{conj:pgc} is true, then this result is equivalent to Alon's theorem. But if not, it is a strictly stronger result than Alon's. Our own interpretation of the result is that it gives yet more evidence for the truth of the positive graph conjecture.

\section{Even-degree vertices in positive graphs}
Following~\cite{Cs13}, we define the \emph{independence polynomial} $I_H(x)$ of a graph $H$ to be
\begin{align*}
    I_H(x) = \sum_{k\geq 0} i_k(H)(-x)^k,
\end{align*}
where $i_k(H)$ denotes the number of independent sets of order $k$ in $H$. To prove~\Cref{thm:even_deg}, the statement that every connected positive graph has a vertex of even degree, we make use of a classical fact about independence polynomials.

\begin{theorem} \label{thm:indpolroot}
    If $H$ is a connected graph with at least one edge, then the root of $I_H(x)$ of  smallest absolute value is real, simple and lies in $(0,1)$. 
\end{theorem}

Apart from the simplicity of the root, \Cref{thm:indpolroot} follows from Theorems 2 and 3 in \cite{FS90}. We refer the interested reader to~\cite{Cs13} for a modern overview as well as a proof of  simplicity, though we note that this was first proved in~\cite{GS00}.

\begin{proof}[Proof of~\Cref{thm:even_deg}]
    Suppose that $H$ is a connected graph with no even-degree vertices.
    Let $\alpha>0$ be a constant to be determined later and let $W$ be the kernel defined by
    \begin{align*}
    W(x,y)=
        \begin{cases}
        0 ~~~~\text{if } x,y\in [0,\alpha], \\ 
        1 ~~~~\text{if } x,y\in (\alpha,1] \text{ and } \\ 
        -1 ~~\text{otherwise.}
        \end{cases}
    \end{align*}
    This is a weighted variant of the looped graph $G=\begin{tikzpicture}[baseline=(x.base)]
    \draw 
    (0,0) node[p](x){}
    (0.4,0.0) node[p](y){};
    \draw (x)--(y);
    \path [-] (y) edge[loop] (y);
    \end{tikzpicture}$ such that $t_H(G)$ counts the (normalised) number of independent sets.
    Then 
    \begin{align*}
        t_H(W) &= \sum_{k=0}^v i_k(H)(-\alpha)^k(1-\alpha)^{v-k} \\
        &= (1-\alpha)^v\sum_{k=0}^v i_k(H)\left(-\frac{\alpha}{1-\alpha}\right)^k
        =(1-\alpha)^vI_H\left(\frac{\alpha}{1-\alpha}\right),
    \end{align*}
    where $v=|V(H)|$. Indeed, the term $i_k(H)(-\alpha)^k(1-\alpha)^{v-k}$ captures the cases where an independent set $I$ of order $k$ is embedded into $[0,\alpha]$ and the other vertices are mapped into $(\alpha,1]$, where the $m=\sum_{v\in I}\deg_H(v)$ crossing edges contribute the weight $(-1)^m = (-1)^k$. This last identity is where we use the fact that all of the vertices have odd degree.
    Now consider the simple root $\beta\in(0,1)$ of $I_H(x)$ given by~\Cref{thm:indpolroot}. By taking $\alpha/(1-\alpha)$ sufficiently close to $\beta$, one can make $I_H\left(\frac{\alpha}{1-\alpha}\right)$ negative.
    Therefore, $H$ is not positive.
\end{proof}

\section{Non-Sidorenko hypergraphs} \label{sec:nonSid}

In this section, we prove our results about non-Sidorenko hypergraphs. However, before getting to the proofs themselves, we need to say a little about quasirandom sequences of hypergraphs and establish a weak counting lemma relative to such sequences.

\subsection{Quasirandom sequences of hypergraphs}

A sequence of graphs $(G_n)$ with $|V(G_n)| \rightarrow \infty$ is \emph{quasirandom}, in the sense of Chung, Graham and Wilson~\cite{CGW89}, if $p=\lim_{n\rightarrow\infty}t_{K_2}(G_n)$ exists and the kernel $U_n:=G_n-p$ satisfies
\begin{align*}
    \int U_n(x,y)g(x)h(y) \rightarrow 0 
\end{align*}
as $n$ tends to infinity for any bounded measurable functions $g, h : [0,1]\rightarrow \mathbb{R}$. In the language of graph limits, this is equivalent to saying that $U_n$ converges to the uniform graphon $p$. In a sense, the whole point of graph limits is that one can work with the limit object rather than with the converging sequence, so one might now expect us to discard the converging sequence. However, for technical reasons which we will say more about below, it will be better to work with the sequential definition when we generalise the notion of quasirandomness to hypergraphs. 

For $r$-graphs with $r\geq 3$, there are many different notions of quasirandomness, each of which is identified with a family $\Q$ of non-empty proper subsets of $[r]$. 
A sequence of $r$-graphs $(\G_n)$ is \emph{$\Q$-quasirandom} if  $p=\lim_{n\rightarrow\infty}t_{\K_r}(\G_n)$ exists and the $r$-kernel $U_n:=\G_n-p$ satisfies
\begin{align}\label{eq:cutnorm}
    \int U_n(x_1,\ldots,x_r)\prod_{Q\in\Q}g_Q(x_Q) \rightarrow 0
\end{align}
as $n\rightarrow \infty$ for any bounded measurable functions $g_Q:[0,1]^Q\rightarrow \mathbb{R}$.
For brevity, we say that a sequence of $r$-kernels $(U_n)$ is \emph{balanced} if their averages tend to zero, i.e., $t_{\K_r}(U_n) \rightarrow 0$ as $n$ tends to infinity. Generalising the definition above, we say that a balanced sequence  $(U_n)$ of $r$-kernels is \emph{$\Q$-quasirandom} if~\eqref{eq:cutnorm} holds for any bounded measurable functions $g_Q:[0,1]^Q\rightarrow \mathbb{R}$. 

We may now remark that the reason we use converging sequences of $r$-kernels $(U_n)$ in our definitions rather than single $r$-kernels is because the space of $r$-kernels is \emph{not} sequentially compact. To compactify the space, one should instead use an appropriately symmetric space of bounded measurable functions with $2^r -2$ variables (see, for example,~\cite{Z15} for more details). However, we make use of $r$-kernels because they are both simpler to use and sufficient for our purposes.

Each different notion of quasirandomness admits a corresponding counting lemma to some extent. For instance (see~\cite{KNRS10}), if $\Q$ consists of all the singleton sets, then every linear $r$-graph $\HH$ has the random-like count $t_{\HH}(\G_n) = (1\pm o(1))t_{\K_r}(\G_n)^{e(\HH)}$ in every $\Q$-quasirandom graph sequence $(\G_n)$. 

What we will need for our results is a weaker condition than saying that there is a random-like count for some particular class of hypergraphs. 
Moreover, we will need that the class of hypergraphs for which our counting result holds is somewhat broader than the class for which the usual counting lemma holds. To say more, we need some terminology. Given a family $\Q$ of non-empty proper subsets of $[r]$, the \emph{closure} of $\Q$ is the family of non-empty subsets $\FF:=\{F\subseteq [r]:\emptyset\neq F\subseteq Q\text{ for some }Q\in\Q\}$. 
We then say that an $r$-graph $\HH$ is \emph{$\Q$-vanishing} if there exists an edge $e^*$ such that the other edges $e\neq e^*$ intersect $e^*$ in such a way that there exists an injective homomorphism from the hypergraph $\{e^*\cap e:e\in E(\HH)\}$ on $e^*$ to the closure of $\Q$ on $[r]$. Our counting lemma is now as follows. 

\begin{lemma}\label{lem:counting}
    If $\HH$ is a $\Q$-vanishing $r$-graph, then $\lim_{n\rightarrow\infty}t_\HH(U_n)=0$ for any $\Q$-quasirandom balanced sequence of kernels $(U_n)$.
\end{lemma}

The following example gives some sense of the difference between this and the usual counting lemma.

\begin{example}
    Let $H$ be the graph obtained by adding a pendant leaf to a $4$-cycle. Then the random-like count $t_H(W)=t_{K_2}(W)^5$ holds if and only if $W$ is quasirandom, i.e., $W$ is constant a.e.
    On the other hand, $H$ is $\Q$-vanishing for $\Q=\{\{1\}\}$, since the leaf edge intersects the other edges on one end only. 
    Hence, in this case, \Cref{lem:counting} is just saying that $\lim_{n \rightarrow \infty} t_H(U_n)=0$ for any balanced sequence of kernels $(U_n)$ with $\int U_n(x,y) g(y) dy \rightarrow 0$ as $n \rightarrow \infty$ for any bounded measurable function $g :[0,1]\rightarrow \mathbb{R}$, which is straightforward to verify.
\end{example}

\begin{proof}[Proof of~\Cref{lem:counting}]
    By relabelling vertices, we may assume that $e^*=[r]$. We may also relabel the vertex set of $\Q$ by using the injective homomorphism $\phi$ from the hypergraph $\{[r]\cap e:e\in E(\HH)\}$ to the closure of $\Q$, so that every $[r]\cap e$, $e\in E(\HH)$, is contained in some $Q\in\Q$.
    
    Fixing all variables other than $x_1,\ldots,x_r$ in $\prod_{i_1i_2\cdots i_r\in E(\HH)}U_n(x_{i_1},\ldots,x_{i_r})$ gives a function of the form
    \begin{align}\label{eq:product}
        U_n(x_1,\ldots,x_r)\prod_{Q\in\Q}g_Q(x_Q),
    \end{align}
    as each $U_n(x_e)$, $e\in E(\HH)$,  is a function of $x_{[r]\cap e}$, which can be seen as a function of $x_{Q}$ for some $Q\in \Q$ containing $[r]\cap e$. Averaging~\eqref{eq:product} over $(x_1,\ldots,x_r)$ then gives the desired conclusion.
\end{proof}

In order to verify the $\Q$-quasirandomness of sequences of $r$-graphs or $r$-kernels, we will make use of the following result. 

\begin{theorem}\label{thm:quasirandom}
    Let $(\G_n)$ be a sequence of $r$-graphs such that $t_{\K_r}(\G_n)$ converges to $p\in [0,1]$ and let $U_n:=\G_n -p$.
    For $\Q\subseteq \mathcal{P}[r] \setminus \{\emptyset, [r]\}$, there exists an $r$-graph $\HH=\HH_\Q$ such that the following are equivalent:
    \begin{enumerate}[(a)]
        \item $t_{\HH}(\G_n)\rightarrow p^{e(\HH)}$ as $n$ tends to infinity;
        \item $t_{\HH}(U_n)\rightarrow 0$ as $n$ tends to infinity;
        \item $(U_n)$ is a $\Q$-quasirandom sequence.
    \end{enumerate}
\end{theorem}

\begin{proof}
    The conditions (a), (b) and (c) are equivalent to $\mathrm{MIN}_{\Q,p}$, $\mathrm{DEV}_{\Q,p}$ and $\mathrm{WDISC}_{\Q,p}$ in~\cite{ACHPS18}, respectively, all of which are shown to be equivalent there.
\end{proof}

The construction of $\HH_\Q$ given in~\cite{ACHPS18} uses a `doubling' operation, but, following~\cite{CL17}, it can also be described explicitly as follows.
Set $q=|\Q|$ and let $M$ be the $r\times q$ `flipped' incidence matrix between $[r]$ and $\Q$, i.e., $M_{ij} = 0$ if $i\in[r]$ is contained in the $j$th set in $\Q$ and otherwise $M_{ij}=1$. 
For each $i\in [r]$, let 
\begin{align*}
    V_i =\{v\in\{+1,-1,0\}^q: v_j=0 \text{ if }M_{ij}=0 \text{ and }v_j\neq 0\text{ otherwise}\}.
\end{align*}
That is, $V_i$ consists of all vectors obtained by possibly flipping the signs of some of the $1$-entries in the $i$th row vector. Then $\HH_\Q$ is the $r$-partite $r$-graph on the disjoint union of all the $V_i$ where $r$ vectors $v^{(1)},\ldots,v^{(r)}$ with $v^{(i)}\in V_i$ form an edge if there exists $v\in\{+1,-1\}^q$ such that $v^{(i)}$ and $v$ agree on all non-zero coordinates of $v^{(i)}$. Observe that each $V_i$ is of order $2^{q-\deg_\Q(i)}$ and there are $2^{q}$ edges in total, each of which corresponds to a vector in $\{+1,-1\}^q$.

\begin{example}\label{ex:C4}
    Let $\Q=\{\{1,2\},\{3\}\}$. Then $\HH_\Q$ is isomorphic to the $3$-graph with vertex set $\{x_1,x_2,y_1,y_2,z_1,z_2\}$ whose edges are $x_iy_iz_j$, $i,j\in\{1,2\}$. 
    Following Erd\H{o}s~\cite{Erd90}, we may define a sequence of $\Q$-quasirandom $3$-graphs by letting $\G_n$ be the $3$-graph with $n$ vertices whose edges are the directed $3$-cycles of the uniform random tournament with $n$ vertices. Then the edge density of $\G_n$ converges to $1/4$ with high probability. Moreover, one can verify the fact that $(\G_n)$ is $\Q$-quasirandom by using~\Cref{thm:quasirandom}. Indeed, the expected number of copies of $\HH_\Q$ in $\G_n$ is $(1\pm o(1))n^{6}(1/2)^8$, where the factor $(1/2)^8$ comes from the probability that, given an orientation of $x_1y_1$ and $x_2y_2$, the eight pairs $x_iz_j$ and $y_iz_j$, $i,j=1,2$, are directed in the unique way that makes four directed 3-cycles.
    Thus, $t_{\HH_Q}(\G_n)$ converges to $(1/4)^4$, proving that $\G_n$ is $\Q$-quasirandom.
\end{example}

\subsection{Non-Sidorenko hypergraphs from quasirandom perturbations}

One way to try to disprove Sidorenko's conjecture for a given hypergraph $\HH$ is by looking at local perturbations of the uniform graphon $1$. That is, we try to show that $t_\HH(1+\varepsilon W)$ can be smaller than $1$ for some $r$-kernel $W$ with $t_{\K_r}(W)=0$. For graphs, it is known that this approach cannot work, since all bipartite graphs are locally Sidorenko, a result of Lov\'asz~\cite{Lov11}. For hypergraphs, the strategy has been more successful, with many of the counterexamples in~\cite{CLS24} coming from this technique. However, we are still far from a classification of those hypergraphs which are locally Sidorenko, in part because the expanded form
\begin{align}\label{eq:expansion}
    t_\HH(1+\varepsilon W) = 1+\sum_{\emptyset\neq\FF\subseteq\HH} t_\FF(W)\varepsilon^{e(\FF)}
\end{align}
contains exponentially many terms that can potentially interact with one another.
We overcome this difficulty here by 
making use of quasirandom perturbations. 

\begin{proposition}\label{prop:nonSidorenko}
    Let $\HH$ be an $r$-graph with a non-positive subgraph $\G$.
    If there exists a sequence $(U_n)_{n=1}^{\infty}$ of $r$-kernels such that $\lim_{n\rightarrow\infty}t_\FF(U_n)=0$ 
    for all subgraphs $\FF\subseteq\HH$ with $0<e(\FF)\leq e(\G)$ except $\G$ and $\lim_{n\rightarrow\infty}|t_{\G}(U_n)|>0$, then $\HH$ is not Sidorenko.
\end{proposition}

\begin{proof}
    Let $\ell:=\lim_{n\rightarrow\infty}t_{\G}(U_n)$ and suppose first that $\ell>0$. Let $W$ be an $r$-kernel such that $t_{\G}(W)<0$ and let $W_n:=W\otimes U_n$. That is, $W_n$ is the tensor product of $W$ and $U_n$, whose key property is that $t_\FF(W_n) = t_\FF(W)t_\FF(U_n)$ for all $r$-graphs $\FF$. Then the expansion formula~\eqref{eq:expansion} gives
    \begin{align*}
        t_{\HH}(1+\varepsilon W_n) = 1+\sum_{\emptyset\neq\FF\subseteq\HH} t_\FF(U_n)t_\FF(W)\varepsilon^{e(\FF)}.
    \end{align*}
    The right-hand side converges to $ 1+ \ell\cdot t_\G(W)\varepsilon^{e(\G)}+O(\varepsilon^{e(\G)+1})$ as $n$ tends to infinity, where $\ell=\lim_{n\rightarrow\infty}t_{\G}(U_n)$.
    On the other hand, $t_{\K_r}(1+\varepsilon W_n) = 1+\varepsilon \cdot t_{\K_r}(U_n)t_{\K_r}(W)$ converges to $1$ as $n$ tends to infinity. Therefore, for $\varepsilon$ sufficiently small, there exists some $n_0$ such that $t_\HH(1+\varepsilon W_n)<t_{\K_r}(1+\varepsilon W_n)^{e(\HH)}$ for all $n\geq n_0$.

    If $\ell<0$, then we can simply observe that the limit of $t_\HH(1+\varepsilon U_n)$ is $ 1+ \ell\cdot \varepsilon^{e(\G)}+O(\varepsilon^{e(\G)+1})$, which is smaller than $1$, the limit of $t_{\K_r}(1+\varepsilon U_n)$, for $\varepsilon$ sufficiently small.
\end{proof}

Our first application of~\Cref{prop:nonSidorenko} is the following result. Together with~\cite[Proposition~1.5]{Sid22}, which states that the grid graph $\G_r$ is not positive for odd $r \geq 3$, this implies~\Cref{thm:grid}.

\begin{theorem} \label{thm:2reg}
    If $\HH$ is a connected linear $r$-graph which is $2$-regular and non-positive, then $\HH$ is not Sidorenko.
\end{theorem}

\begin{proof}
    For convenience in the proof, we will work over the measure space $\{-1,+1\}^r$ with uniform measure instead of over $[0,1]^r$. 
    Let $U(x_1,\ldots,x_r)=x_1x_2\cdots x_r$. Then 
    \begin{align*}
        t_\FF(U) = \EE\left[\prod_{v\in V(\FF)} x_v^{\deg_\FF(v)}\right],
    \end{align*}
    which evaluates to $1$ if all of the degrees are even and $0$ otherwise. Thus, $ t_\FF(U) =0$ for any non-empty proper subgraph $\FF$ of $\HH$, whereas $t_\HH(U)=1$.
    Taking $\G=\HH$ and $U_n=U$ for all $n$  in~\Cref{prop:nonSidorenko} then concludes the proof.
\end{proof}

In what follows, we will apply~\Cref{prop:nonSidorenko} by first picking a family $\Q$ of subsets of $[r]$ for which all proper subgraphs of $\HH$ are $\Q$-vanishing. 
If we then take a $\Q$-quasirandom sequence $(U_n)$ of $r$-kernels, \Cref{lem:counting} implies that $\lim_{n\rightarrow\infty}t_\HH(U_n)=0$ for all proper subgraphs $\FF$ of $\HH$. Therefore, in order to conclude from~\Cref{prop:nonSidorenko} that $\HH$ is not Sidorenko, it only remains to verify that $t_\HH(U_n)$ converges to a non-zero number, for which we can often execute a hands-on calculation.

In the proof of~\Cref{thm:2reg}, the function $U$ is $(r-1)$-codegree-regular, meaning that 
\begin{align*}
    \int U(x_1,\ldots,x_r)g(x_{[r]\setminus i})=0
\end{align*}
for every $i\in [r]$ and every bounded measurable $g:[0,1]^{r-1}\rightarrow\mathbb{R}$, where $x_{[r]\setminus i}$ denotes the 
vector formed from $x$ by removing the $i$th coordinate. 
Therefore, $U_n=U$ is $\Q$-quasirandom for $\Q=\{[r-1]\}$, 
enough to imply that $\FF$ is $\Q$-vanishing for every proper subgraph of $\HH$. However, crucially, $\HH$ is not  itself $\Q$-vanishing, so we could 
apply~\Cref{prop:nonSidorenko} to conclude that $\HH$ is not Sidorenko.
If we had instead taken $\Q=\{\{i\}:i\in [r]\}$, then $\Q$-quasirandom means weak quasirandomness in the sense of~\cite{KNRS10}, which is too strong because then even $\HH$ itself is $\Q$-vanishing. 

Our proof of~\Cref{thm:cycle}, the statement that $\C_{\ell}^{(r)}$ is not Sidorenko for any odd $r \geq 3$, makes use of sequences of  quasirandom $r$-graphs defined using random higher-order tournaments that were introduced by Reiher, R\"odl and Schacht~\cite{RRS18}. 
As a gentle introduction to these ideas, we first give a new proof of the particular case $r=3$ of~\Cref{thm:cycle} proved in~\cite{CLS24}, as it only uses random tournaments rather than their higher-order generalisations. Since $\C_\ell^{(3)}$ is only tripartite if $\ell$ is a multiple of $3$, it will suffice to consider the case where $\ell = 3k$.

\begin{proposition}
    If $k \geq 2$, then the tight cycle $\C_{3k}^{(3)}$ is not Sidorenko.
\end{proposition}

\begin{proof}
    We will use the $3$-graph sequence $(\G_n)$ described in \Cref{ex:C4}, which is $\Q$-quasirandom with $\Q=\{\{1,2\},\{3\}\}$. 
    Suppose that the vertices of $\C_{3k}^{(3)}$ have been labelled by the elements of $[3k]$ in such a way that $\{i,i+1,i+2\}$ is an edge for all $i \in [3k]$, where  addition is taken modulo $3k$. 
    Given a non-empty proper subgraph $\FF$ of $\C_{3k}^{(3)}$, let $\{i,i+1,i+2\}$ be a missing edge such that $e=\{i-1,i,i+1\}\in E(\FF)$.
    Then all other edges intersect $e$ on a subset of either $\{i-1,i\}$ or $\{i+1\}$ and, hence, $\FF$ is $\Q$-vanishing.  By~\Cref{lem:counting}, if we set $U_n = \G_n - 1/4$, it follows that $\lim_{n\rightarrow\infty}t_\FF(U_n) = 0$. 

    Setting $\HH=\C_{3k}^{(3)}$, we claim that $\lim_{n\rightarrow\infty}t_\HH(U_n)$ is non-zero. As
    \begin{align*}
        t_\HH(\G_n) = t_\HH(U_n+1/4)  =\sum_{\FF\subseteq\HH}t_\FF(U_n)(1/4)^{3k-e(\FF)} \rightarrow \lim_{n \rightarrow \infty} t_\HH(U_n) + (1/4)^{3k},
    \end{align*}
    it will suffice to prove that $t_\HH(\G_n)$ converges to a number distinct from $(1/4)^{3k}$. 
    
    Suppose that 
    $\phi:[3k]\rightarrow [n]$ is an injective homomorphism from $\C_{3k}^{(3)}$ to $\G_n$. 
    In the random tournament that produces $\G_n$, if $(\phi(1),\phi(2))$ is a directed edge, then every $(\phi(i),\phi(i+1))$ and $(\phi(i+2),\phi(i))$ must also be directed edges for all $i\in [3k]$. This event occurs with probability $(1/2)^{6k}$.
    Adding in the other case where $(\phi(2),\phi(1))$ is a directed edge in the tournament and using standard concentration inequalities, we see that $t_\HH(\G_n)=(2\pm o(1))(1/4)^{3k}$ with high probability. Thus, $t_\HH(\G_n)$ converges to $2(1/4)^{3k}$ rather than $(1/4)^{3k}$.

    Since $\C_{3k}^{(3)}$ is non-positive by~\Cref{cor:odd}, an application of~\Cref{prop:nonSidorenko} with $U_n$ and $\G=\HH$ completes the proof.
\end{proof}

Following~\cite{RRS18}, we define an \emph{$(r-1)$-uniform tournament} (or \emph{$(r-1)$-tournament}) to be the complete $(r-1)$-graph $\K_n^{(r-1)}$ together with an orientation $\sigma(T)$ of each edge $T\in\binom{[n]}{r-1}$. Informally, an orientation is just an assignment of $\pm 1$ to each edge. More formally, we map the set of all even permutations of $T$ (with the elements of $T$ written in increasing order corresponding to the identity permutation) to either $+1$ or $-1$ and the odd permutations of $T$ to the opposite number.

We may also assign $\pm 1$-values to each subset $S = T\setminus\{v\}$ of order $r-2$ in $T$ by considering the permutation of $T$ that orders the elements of $S$ in the same order as $[n]$ and places $v$ at the end and assigning $S$ the same $\pm 1$-value as this permutation. 
We call this the \emph{$T$-sign} of the $(r-2)$-subset $S$ and denote it by $\sigma_T$. For example, if $r=3$, the $T$-sign of a vertex $v$ indicates whether the edge $T$ is directed towards $v$ or away from it.

Given a set $R\subseteq [n]$ of order $r$ and $S\in \binom{R}{r-2}$, the \emph{$R$-weight} of $S$ is given by the sum of the $T$-signs of $S$ over all $T\in \binom{R}{r-1}$ containing~$S$. As there are only two choices for such a $T\in \binom{R}{r-1}$, the weight is either $+2$, $-2$ or $0$.
When $r=3$, the $R$-weight of a vertex counts the outdegree minus the indegree of a vertex in the given orientation of the edges of the triangle induced on $R$.

Let $\G(\TT_n)$ be the $r$-graph whose edges are those $r$-subsets $R$ of $[n]$ such that the $R$-weight of each $(r-2)$-subset $S\subseteq R$ is zero. Note that this clearly generalises the construction using directed $3$-cycles that we employed in the $r=3$ case. 
We now prove some facts about $\G(\TT_n)$ if $\TT_n$ is the uniform random $(r-1)$-tournament where the orientation of each edge is chosen uniformly at random.

\begin{lemma}\label{lem:expectation}
    If $\TT_n$ is the uniform random $(r-1)$-tournament, then $\G(\TT_n)$ has the following properties:
    \begin{enumerate}[(i)]
        \item $\EE_{\TT_n}[t_{\K_r}(\G(\TT_n))]=2^{1-r} + o(1)$;
        \item if $\HH$ is an $r$-graph for which there is an ordering $e_1, e_2, \dots, e_m$ of the edges such that $\{e_i\cap e_j: j < i\}$ contains at most one set of order $r-1$ for each $e_i$, then $\EE_{\TT_n}[t_{\HH}(\G(\TT_n))]=2^{(1-r)e(\HH)} + o(1)$. 
    \end{enumerate}
\end{lemma}

\begin{proof}
    Once the orientation of one $(r-1)$-subset $T\subseteq R$ is fixed, there is a unique way to give orientations to the other $(r-1)$-subsets so that $R\in E(\G(\TT_n))$. Indeed, every $(r-1)$-subset of $R$ is obtained by exchanging one element $x$ of $T$ with the unique element $y\in R\setminus T$. 
    Since the $(r-2)$-subset $R\setminus\{x,y\}$ already has a fixed $T$-sign, $T'=T\setminus\{x\}\cup\{y\}$ must assign the opposite $T'$-sign to $R\setminus\{x,y\}$, so the extension, if it exists, must be unique.
    The existence of an extension follows from the simple fact that there is an orientation of all $T\in\binom{R}{r-1}$ such that the $R$-weight of each $S\in\binom{R}{r-2}$ vanishes. Indeed, if $R = [r]$, one may take the orientation where $\sigma(1 \, 2 \, \dots \, \hat{i} \, \dots \, r) = (-1)^{r-i}$ or its negative.

    The $(r-1)$-uniform tournament decides the orientation of each $(r-1)$-subset independently at random, so, once an orientation of any particular $(r-1)$-set $T$ is fixed, the probability that an $r$-set $R\supseteq T$ becomes an edge is $2^{1-r}$, since we need that each of the $r-1$ other $(r-1)$-subsets of $R$ receives the `correct' orientation. Therefore, the expectation of $t_{\K_r}(\G(\TT_n))$ is $2^{1-r} + o(1)$, where the $o(1)$-term accounts for degeneracies.

    To prove (ii), it is enough to fix a set of order $|V(\HH)|$ in $[n]$ and show that the probability it spans a labelled copy of $\HH$ is $2^{(1-r)e(\HH)}$.
    We will use induction on $e(\HH)$ to show this, noting that the base case follows from (i). 
    Consider $\HH':=\HH\setminus e_m$. By induction, the probability that $\HH'$ appears on $V(\HH)$ is $2^{(1-r)(e(\HH)-1)}$. 
    If $e_m$ does not intersect any other edges on a set of order $r-1$, then, as above, the probability that $e_m$ is an edge of $\G(\TT_n)$ is $2^{1-r}$ independently of all other edges, completing the induction in this case. 
    If instead $e_m$ has an intersection $e_m\cap e_j$ of order $r-1$ with some other edge $e_j$, then only the orientation of this intersection is fixed, while the orientation of the remaining $(r-1)$-subsets of $e_m$ is yet to be determined. Therefore, the probability that $e_m$ is an edge of $\G(\TT_n)$ is again 
    $2^{1-r}$ independently of all other edges, completing the induction and the proof.
\end{proof}

In order to carry out our scheme, we need to verify that the sequence of $r$-graphs $(\G(\TT_n))$ arising from a sequence $(\TT_n)$ of uniform random $(r-1)$-tournaments is  $\Q$-quasirandom with high probability for some appropriate choice of $\Q$.

\begin{corollary}\label{cor:q_quasi}
    If $(\TT_n)$ is a sequence of uniform random $(r-1)$-tournaments with $|V(\TT_n)| = n$, then, with high probability, $(\G(\TT_n))$ is $\Q$-quasirandom for $\Q=\{[r-1]\}\cup \binom{[r]}{r-2}$.
\end{corollary}
\begin{proof}
    Let $\HH=\HH_Q$ be the hypergraph given by~\Cref{thm:quasirandom}. Enumerate $\Q$ as $\{Q_1,\dots,Q_q\}$. Then, as described below~\Cref{thm:quasirandom}, $\HH$ can be taken to be the $r$-partite graph with $r$-partition $V_1\cup\cdots\cup V_r$, where
    \begin{align*}
        V_i =\{v\in\{+1,-1,0\}^{q}: v_j=0 \text{ if }i\in Q_j \text{ and }v_j\neq 0\text{ otherwise}\}
    \end{align*}
    and each edge $e\in E(\HH)$ consists of $r$ vectors $v^{(1)},\ldots,v^{(r)}$, $v^{(i)}\in V_i$, that agree on all non-zero coordinates.
    In particular, simultaneously flipping the sign of the same set of coordinates in each of $v^{(1)},\ldots,v^{(r)}$  
    maps an edge $e$ to another edge $e'$. If two edges $e$ and $e'$ share $r-1$ elements, then this means that flipping the sign of some coordinate $j$ leaves $r-1$ of the vectors unchanged. Thus, the corresponding entries $v^{(i_1)}_j,\ldots,v^{(i_{r-1})}_j$ must all be zero.
    But this can only happen for the coordinate which corresponds to $[r-1]$. Hence, $\HH$ satisfies the condition of part (ii) of~\Cref{lem:expectation}, so $\EE_{\TT_n}[t_{\HH}(\G(\TT_n))]=2^{(1-r)e(\HH)} + o(1)$. By standard concentration inequalities, $(t_{\HH}(\G(\TT_n)))$ converges to $2^{(1-r)e(\HH)}$ with high probability, so~\Cref{thm:quasirandom} shows that $(\G(\TT_n))$ is $\Q$-quasirandom.
\end{proof}

We are finally in a position to prove~\Cref{thm:cycle}, which we again recall states that $\C_{\ell}^{(r)}$ is not Sidorenko for any odd $r \geq 3$ and $\ell > r$. Note that $\C_{\ell}^{(r)}$ is only $r$-partite if $\ell$ is a multiple of $r$, so we can assume without loss of generality that we are in that case.

\begin{proof}[Proof of~\Cref{thm:cycle}]
    Let $(\G_n):=(\G(\TT_n))$ be the $\Q$-quasirandom sequence of $r$-graphs given by~\Cref{cor:q_quasi} with $\Q=\{[r-1]\}\cup \binom{[r]}{r-2}$. 
    To check that any proper non-empty subgraph of $\C_{rk}^{(r)}$ is $\Q$-vanishing, it is enough to look at an edge $e\in E(\C_{rk}^{(r)})$ that intersects a missing edge on a subset of order $r-1$. But the hypergraph on $e$ consisting of the possible intersections $e\cap e'$ of $e$ with other edges $e'$ in the subgraph is, after possibly relabelling, contained in the closure of $\Q$, confirming that the subgraph is $\Q$-vanishing.
    
    Write $\C=\C_{rk}^{(r)}$ and $p=2^{1-r}$ for brevity. 
    Since~\Cref{cor:odd} implies that $\C$ is not positive, in order to
    apply~\Cref{prop:nonSidorenko}, it only remains to prove that $t_{\C}(U_n)$ does not converge to zero for $U_n=\G_n - p$. Since
    \begin{align*}
        t_{\C}(\G_n) - p^{e(\C)} = t_{\C}(U_n +p)-p^{e(\C)}=  \sum_{\emptyset\neq \FF\subseteq \C} t_\FF(U_n)p^{e(\C)-e(\FF)}
    \end{align*}
    and every term in the sum except $t_{\C}(U_n)$ converges to zero,
    $t_{\C}(U_n)\rightarrow 0$ if and only if $t_{\C}(\G_n) \rightarrow p^{e(\C)}$.

    It therefore suffices to prove that $t_{\C}(\G_n)$ does not converge to $p^{e(\C)}$. 
    To this end, setting $\ell:=rk$ for brevity, let us 
    compute the probability that there is a copy of $\C$ on $[\ell]$ with $\{i,i+1,\ldots,i+r-1\}$ an edge for every $i\in [\ell]$,  where addition is taken modulo $\ell$. 
    In particular, $e=[r]$ is one of the edges of our possible copy of $\C$.
    Suppose that we have exposed all of the $(r-1)$-subsets that are contained in some other edge $e'\neq e$ of our possible copy of $\C$. Then the probability that we have created a copy of $\C\setminus e$ is $2^{(1-r)(\ell -1)}$ by part (ii) of~\Cref{lem:expectation}. 
    Note that, at this point, the $(r-1)$-subsets contained in $e$ that have not yet been given an orientation are those containing both $1$ and $r$.

    Let $T_{j,k}:=\{j,j+1,\dots,j+k-1\}$ be the $k$ consecutive integers starting from $j$, where addition is taken modulo $\ell$.
    Recall that $\sigma(T)$ denotes the orientation of an $(r-1)$-set $T$ and $\sigma_T(S)$ the $T$-sign of $S$ of an $(r-2)$-subset $S$ of $T$.
    If $\sigma(T_{1,r-1})=+1$, then $\sigma_{T_{1,r-1}}(T_{1,r-2})=+1$, so, since $\{\ell, 1 \dots, r-1\}$ is an edge of $\G_n$, we must have $\sigma_{T_{\ell, r-2}}(T_{1,r-2})=-1$, which in turn implies that $\sigma(T_{\ell,r-2})=(-1)^{r-1}$. Repeating this procedure, we see that $\sigma(T_{\ell-1,r-3})=(-1)^{2(r-1)}$, $\sigma(T_{\ell-2,r-4})=(-1)^{3(r-1)}$, \dots, $\sigma(T_{2,r+1})=(-1)^{(\ell-1)(r-1)}$. But then $\sigma_{T_{1,r-1}}(T_{2,r-1})=(-1)^{r-2}$ and $\sigma_{T_{2,r}}(T_{2,r-1})=(-1)^{(\ell-1)(r-1)}$ and, using that $\ell = rk$, we see that these have opposite signs, a necessary and sufficient condition for us to be able to orient the remaining $(r-1)$-subets of $e$. Therefore, if we get a consistent orientation for the $r-2$ remaining $(r-1)$-subsets of $e$, which occurs with probability $2^{2-r}$, we get a copy of $\C$. In total, multiplying by $2^{(1-r)(\ell -1)}$, the probability of a labelled copy of $\C$ on vertex set $[\ell]$ is $2^{1+(1-r)\ell}$. By standard concentration inequalities, we see that $t_\C(\G_n) = 2^{1+(1-r)\ell} + o(1)$ with high probability and this is indeed larger than the random bound $p^{e(\C)} = 2^{(1-r)\ell}$.
\end{proof}

The results of this section also extend to grids and tight cycles of even uniformity at least $4$, in the sense that if any of these hypergraphs can be shown to be non-positive, then our arguments imply that they are also non-Sidorenko. If one assumes the positive graph conjecture, then one can check that these hypergraphs are indeed non-positive. We leave it as an open problem to find a proof of non-positivity that does not make use of this considerable assumption.

\section{Graph codes}
Let $\FFF_2^{K_n}$ be the vector space of all graphs on vertex set $[n]$, where the edge set of the sum of two graphs is given by the symmetric difference of the summand edge sets. For a function $f\colon \FFF_2^{K_n}\rightarrow \RR$, its \emph{Fourier transform} $\ft{f}\colon \FFF_2^{K_n} \rightarrow \mathbb{C}$ is defined by 
\begin{align*}
     x\mapsto \frac{1}{\vert \FFF_2^{K_n}\vert} \sum_{y\in \FFF_2^{K_n}} (-1)^{x^Ty} f(y)=\frac{1}{\vert \FFF_2^{K_n}\vert} \sum_{y\in \FFF_2^{K_n}} (-1)^{\sum_{e\in \binom{[n]}{2}} x_ey_e} f(y).
\end{align*}
Recall, from the introduction, that $d_H(n)$ is the maximum proportion of graphs on $[n]$ in an \emph{$H$-code}, a family without two members whose symmetric difference is a copy of $H$. The following result gives a Fourier bound on this graph code density.

\begin{lemma} \label{lem:fourier}
Let $H$ be a graph and, for $n\in \mathbb{N}$, let $B_n$ be the set of copies of $H$ in $\FFF_2^{K_n}$. Then $d_H(n)<-(|\FFF_2^{K_n}|/|B_n|) \min_{y\in \FFF_2^{K_n}}\ft{1_B}(y)$.
\end{lemma}

\begin{proof}
    Let $A\subseteq\FFF_2^{K_n}$ be an $H$-code and let $B:=B_n$ be the set of copies of $H$. 
    We can express the fact that $A+A$ does not intersect $B$ as saying that the inner product $\langle 1_A*1_A,1_B\rangle=\EE_x 1_A*1_A(x)1_B(x)$ is 0, where $f*g$ denotes the convolution
\begin{align*}
f*g(x)=\EE_{y\in \FFF_2^{K_n}} f(y)g(x-y)=\EE_{y\in \FFF_2^{K_n}} f(y)g(x+y).
\end{align*}
Using Parseval's identity and the well known fact that the Fourier transform of $f*g$ is $\ft{f}\cdot\ft{g}$, we deduce that\footnote{Note that the inner product in Fourier space has no normalisation.}
\begin{align}\label{eq:graph-codes-additive-inner-product}
\langle 1_A*1_A,1_B\rangle = \langle \ft{1_A}^2,\ft{1_B}\rangle=\sum_{x\in \FFF_2^{K_n}} \ft{1_A}(x)^2\ft{1_B}(x)\geq \ft{1_A}(0)^2\ft{1_B}(0)+\Big(\min_{y\in \FFF_2^{K_n}} \ft{1_B}(y)\Big)\sum_{\substack{x\in \FFF_2^{K_n}\\x\neq 0}} \ft{1_A}(x)^2.
\end{align}
Now let $\alpha:=\vert A\vert/\vert \FFF_2^{K_n}\vert$, $\beta:=\vert B\vert/\vert \FFF_2^{K_n}\vert$ and $\gamma = \min_{y\in \FFF_2^{K_n}} \ft{1_B}(y)$. We then have $\ft{1_A}(0)^2\ft{1_B}(0)=\alpha^2\beta$ and $\sum_{x\in \FFF_2^{K_n}} \ft{1_A}(x)^2=\alpha$. Hence, from \eqref{eq:graph-codes-additive-inner-product}, we can infer that $\alpha^2\beta + \gamma\alpha<0$, meaning that $\alpha<-\gamma/\beta$, which proves the claim.
\end{proof}

We now give a characterisation, which may be interesting in its own right, of whether or not a graph $H$ is positive in terms of the Fourier coefficients of $1_{B_n}$, the indicator function of the set of copies of $H$ in $\FFF_2^{K_n}$.

\begin{proposition}\label{prop:graph-codes-positivity}
Let $H$ be a graph and let $B:=B_n$ be the set of copies of $H$ in $\FFF_2^{K_n}$. If $H$ is positive, then there exists $C>0$ such that $\ft{1_B}(x)>-Cn^{v(H)-1}/\vert \FFF_2^{K_n}\vert$ for all $n\in \NN$ and all $x\in \FFF_2^{K_n}$. Conversely, if $H$ is not positive, then there exists $c>0$ such that, for all sufficiently large $n\in \NN$, there is $x\in \FFF_2^{K_n}$ with $\ft{1_B}(x)<-cn^{v(H)}/\vert \FFF_2^{K_n}\vert$.
\end{proposition}

\begin{proof}
Let $\beta:=\vert B\vert/\vert \FFF_2^{K_n}\vert$ and suppose that there exist $n\in \NN$ and $x\in \FFF_2^{K_n}$ such that  $\ft{1_B}(x)<-2v(H)^{v(H)}\beta/n$. Consider the signed graph
\begin{align*}
    W\colon \binom{[n]}{2} \rightarrow \{\pm 1\}, \qquad uv\mapsto \begin{cases}
        -1&\text{if $uv$ is an edge of $x$,}\\
        1&\text{if $uv$ is not an edge of $x$.}
    \end{cases}
\end{align*}
Writing $\mathrm{Aut}(H)$ for the set of automorphisms of $H$, we have that 
\begin{align*}
t_H(W)&\leq \frac{\vert\mathrm{Aut}(H)\vert}{n^{v(H)}} \Big( \vert \{y\in B:\vert E(x)\cap E(y)\vert \text{ is even} \}\vert \\&\hspace*{2.0cm}- \vert\{y\in B:\vert E(x)\cap E(y)\vert \text{ is odd}\}\vert \Big)+v(H)^{v(H)}/n\\
&\leq \frac{1}{2\vert B \vert} \sum_{y\in B}(-1)^{x^Ty}+v(H)^{v(H)}/n,
\end{align*}
which is $\ft{1_B}(x)/2\beta+v(H)^{v(H)}/n<0$, showing that $H$ is not positive.

Conversely, if $H$ is not positive, it follows from standard results in the theory of graph limits that there exists a weighted graph $W$ on vertex set $[n]$ for some $n \in \mathbb{N}$ with weights $w_{ij}\in [-1,1]$ and $w_{ii}=0$ for $i,j\in [n]$ such that $t_H(W)<0$. 
For each $k\in \NN$, consider the random graph $x:=x_k$ with vertex set $[kn]$ such that 
each edge $uv$ is included with probability $(1-w_{\pi(u)\pi(v)})/2$, where $\pi(u)$ is the residue of $u$ after division by $n$. 
Let us denote by $X_{uv}$ the indicator of the event that the edge $uv$ is present. Since all of these events are independent, the expectation of $\ft{1_{B_{kn}}}(x)$ is 
\begin{align*}
\EE\big[\ft{1_{B_{kn}}}(x)\big]
&=\frac{\beta}{\vert B_{kn}\vert} \sum_{y\in B_{kn}} \prod_{uv\in E(y)} \EE\big[(-1)^{X_{uv}}\big]
=\frac{\beta}{\vert B_{kn}\vert} \sum_{y\in B_{kn}} \prod_{uv\in E(y)} w_{\pi(u)\pi(v)}.
\end{align*}
Because the proportion of non-injective maps $V(H)\rightarrow [kn]$ becomes small for large $k$, the expected value of $\ft{1_{B_{kn}}}(x)/\beta$ converges to $t_H(W)<0$, giving the required result.
\end{proof}

Combining~\Cref{lem:fourier} and the first part of~\Cref{prop:graph-codes-positivity}, we see that if $H$ is positive and $n$ is sufficiently large, then
\[d_H(n)<-\frac{\vert \FFF_2^{K_n}\vert}{\vert B_n\vert} \min_{y\in \FFF_2^{K_n}}\ft{1_B}(y) < \frac{\vert \FFF_2^{K_n}\vert}{\vert B_n\vert} \frac{C n^{v(H)-1}}{\vert \FFF_2^{K_n}\vert} \leq \frac{2 C \vert\mathrm{Aut}(H)\vert}{n},\]
concluding the proof of Theorem~\ref{thm:code}.

\medskip

\noindent{\bf Acknowledgements.} The authors are grateful to Sasha Sidorenko for helpful comments on an early draft of this paper. Part of this work was carried out while the third author visited the second author at Yonsei University under the support of the latter's Samsung STF Grant.

\bibliographystyle{plainurl}
\bibliography{references}

\begin{thebibliography}{10}

\bibitem{ACHPS18}
Elad Aigner-Horev, David Conlon, Hi\^{e}p H{\`a}n, Yury Person, and Mathias Schacht.
\newblock Quasirandomness in hypergraphs.
\newblock {\em Electron. J. Combin.}, 25:Paper No. 3.34, 22 pp., 2018.

\bibitem{A24}
Noga Alon.
\newblock Graph-codes.
\newblock {\em European J. Combin.}, 116:Paper No. 103880, 7 pp., 2024.

\bibitem{AGKMS23}
Noga Alon, Anna Gujgiczer, J\'{a}nos K\"{o}rner, Aleksa Milojevi\'{c}, and G\'{a}bor Simonyi.
\newblock Structured codes of graphs.
\newblock {\em SIAM J. Discrete Math.}, 37:379--403, 2023.

\bibitem{BRW24}
Grigoriy Blekherman, Annie Raymond, and Fan Wei.
\newblock Undecidability of polynomial inequalities in weighted graph homomorphism densities.
\newblock {\em Forum Math. Sigma}, 12:Paper No. e40, 2024.

\bibitem{ACHLL12}
Omar~Antol\'{\i}n Camarena, Endre Cs\'{o}ka, Tam\'{a}s Hubai, G\'{a}bor Lippner, and L\'{a}szl\'{o} Lov\'asz.
\newblock Positive graphs.
\newblock {\em European J. Combin.}, 52B:290--301, 2016.

\bibitem{CGW89}
Fan R.~K. Chung, Ronald~L. Graham, and Richard~M. Wilson.
\newblock Quasi-random graphs.
\newblock {\em Combinatorica}, 9:345--362, 1989.

\bibitem{CL17}
David Conlon and Joonkyung Lee.
\newblock Finite reflection groups and graph norms.
\newblock {\em Adv. Math.}, 315:130--165, 2017.

\bibitem{CL21}
David Conlon and Joonkyung Lee.
\newblock Sidorenko's conjecture for blow-ups.
\newblock {\em Discrete Anal.}, 2021:Paper No. 2, 13 pp., 2021.

\bibitem{CLS24}
David Conlon, Joonkyung Lee, and Alexander Sidorenko.
\newblock Extremal numbers and {S}idorenko's conjecture.
\newblock To appear in Int. Math. Res. Not., 2024.

\bibitem{Cs13}
P\'eter Csikv\'ari.
\newblock Note on the smallest root of the independence polynomial.
\newblock {\em Combin. Probab. Comput.}, 22:1--8, 2013.

\bibitem{Erd90}
Paul Erd\H{o}s.
\newblock Problems and results on graphs and hypergraphs: similarities and differences.
\newblock In {\em Mathematics of {R}amsey theory}, volume~5 of {\em Algorithms Combin.}, pages 12--28. Springer, Berlin, 1990.

\bibitem{FMS20}
Asaf Ferber, Gweneth McKinley, and Wojciech Samotij.
\newblock Supersaturated sparse graphs and hypergraphs.
\newblock {\em Int. Math. Res. Not.}, 2020:378--402, 2020.

\bibitem{FS90}
David~C. Fisher and Anita~E. Solow.
\newblock Dependence polynomials.
\newblock {\em Discrete Math.}, 3:251--258, 1990.

\bibitem{GS00}
Massimiliano Goldwurm and Massimo Santini.
\newblock Clique polynomials have a unique root of smallest modulus.
\newblock {\em Inform. Process. Lett.}, 75:127--132, 2000.

\bibitem{GL}
Tim Gowers and Jason Long.
\newblock Personal communication.
\newblock 2018.

\bibitem{HN11}
Hamed Hatami and Serguei Norine.
\newblock Undecidability of linear inequalities in graph homomorphism densities.
\newblock {\em J. Amer. Math. Soc.}, 24:547--565, 2011.

\bibitem{KNRS10}
Yoshiharu Kohayakawa, Brendan Nagle, Vojt{\v{e}}ch R{\"o}dl, and Mathias Schacht.
\newblock Weak hypergraph regularity and linear hypergraphs.
\newblock {\em J. Combin. Theory Ser. B}, 100:151--160, 2010.

\bibitem{L08}
L\'{a}szl\'{o} Lov\'{a}sz.
\newblock Graph homomorphisms: Open problems.
\newblock Manuscript available at http://www.cs.elte.hu/~lovasz/problems.pdf, 2008.

\bibitem{Lov11}
L\'{a}szl\'{o} Lov\'{a}sz.
\newblock Subgraph densities in signed graphons and the local {S}imonovits--{S}idorenko conjecture.
\newblock {\em Electron. J. Combin.}, 18:Paper 127, 21 pp., 2011.

\bibitem{RRS18}
Christian Reiher, Vojt{\v{e}}ch R{\"o}dl, and Mathias Schacht.
\newblock On a generalisation of {M}antel’s theorem to uniformly dense hypergraphs.
\newblock {\em Int. Math. Res. Not.}, 16:4899--4941, 2018.

\bibitem{Sid92}
Alexander Sidorenko.
\newblock Inequalities for functionals generated by bipartite graphs.
\newblock {\em Discrete Math. Appl.}, 2:489--504, 1992.

\bibitem{Sid93}
Alexander Sidorenko.
\newblock A correlation inequality for bipartite graphs.
\newblock {\em Graphs Combin.}, 9:201--204, 1993.

\bibitem{Sid22}
Alexander Sidorenko.
\newblock On positive hypergraphs.
\newblock {\em European J. Combin.}, 106:Paper No. 103574, 5 pp., 2022.

\bibitem{V24}
Leo Versteegen.
\newblock Upper bounds for linear graph codes.
\newblock Preprint available at ar{X}iv:2310.19891 [math.CO].

\bibitem{Z15}
Yufei Zhao.
\newblock Hypergraph limits: a regularity approach.
\newblock {\em Random Structures Algorithms}, 47:205--226, 2015.

\end{thebibliography}

\end{document}